\documentclass[10pt]{amsart}
\usepackage{graphicx}
\usepackage{amsmath}
\usepackage{amsthm,amsfonts,amssymb,mathrsfs,amscd,amstext,amsbsy}
\usepackage{enumerate}
\usepackage[all]{xy}
\usepackage{tikz}
\usepackage{color}

\usetikzlibrary{patterns}

\newtheorem{theorem}{Theorem}[section]

\newtheorem{lem}[theorem]{Lemma}

\newtheorem{obs}[theorem]{Observation}

\newtheorem{claim}[theorem]{Claim}
\newtheorem{thm}[theorem]{Theorem}
\newtheorem{problem}[theorem]{Problem}

\def\qed{\hfill \ifhmode\unskip\nobreak\fi\quad\ifmmode\Box\else$\Box$\fi\\ }

\begin{document}
\title{On the number of $r$-matchings in a Tree}

\author{Dong Yeap Kang, Jaehoon Kim, Younjin Kim and Hiu-Fai Law}
 
\address[Dong Yeap Kang]{Department of Mathematical Sciences, KAIST, 291 Daehak-ro Yuseong-gu Daejeon, 305-701 South Korea}
\address[Jaehoon Kim]{Department of Mathematical Sciences, KAIST, 291 Daehak-ro Yuseong-gu Daejeon, 305-701 South Korea / School of Mathematics, University of Birmingham, Edgbaston, Birmingham B15 2TT, United Kingdom}
\address[Younjin Kim]{Department of Mathematical Sciences, KAIST, 291 Daehak-ro Yuseong-gu Daejeon, 305-701 South Korea}
\address[Hiu-Fai Law]{Fachbereich Mathematik, Universit\"at Hamburg, Germany}
\email{dynamical@kaist.ac.kr}
\email{kimJS@bham.ac.uk, mutualteon@gmail.com}
\email{younjin@kaist.ac.kr}
\email{hiufai.law@gmail.com}
\thanks{ The first author and second author were partially supported by Basic Science Research Program through the National Research Foundation of Korea(NRF) funded by the Ministry of
Science, ICT \& Future Planning (2011-0011653). \\
The second author was also partially supported by the European Research Council under the European Union's Seventh Framework Programme (FP/2007--2013) / ERC Grant Agreements no. 306349 (J.~Kim). \\
The third author was supported by Basic Science Research Program through the National Research Foundation of Korea(NRF) funded by the Ministry of
Science, ICT \& Future Planning (2013R1A1A3010982).\\
Corresponding author: Younjin Kim}

\date{\today}
\begin{abstract}
An $r$-matching in a graph $G$ is a collection of edges in $G$ such that the distance between any two edges is at least $r$. A $2$-matching is also called an induced matching. In this paper, we estimate the maximum number of $r$-matchings in a tree of fixed order.  We also prove that the $n$-vertex path has the maximum number of induced matchings among all $n$-vertex trees.
 \end{abstract}

\maketitle
\section {Introduction}\label{intro}

A set of vertices in a graph $G$ is called {\it independent} if the set does not induce any edges. The number of independent sets in $G$, denoted by $i(G)$, was first studied by Prodinger and Tichy~\cite{PT82}, who proved that among trees of the same order, the star and the path attain the maximum and the minimum of the parameter, respectively. 
Since then many extremal results on $i(G)$ over various families of graphs have been obtained, and connections to mathematical physics also have been explored \cite{SS05}. A closely related parameter is the number of matchings in $G$. Another result proved in \cite{PT82} is that paths and stars, among trees of a fixed order, are the unique trees that attain maximum and minimum of the number of matchings, respectively.  \\

Independent sets and matchings can be generalized in the following way as proposed in \cite{AHK11}, \cite{Law12}. Given a graph $F$, an {\it $F$-matching} is a subgraph of $G$ whose components are isomorphic to $F$. An {\it induced $F$-matching} is an $F$-matching such that there are no induced edges between the copies of $F$. 
To generalize it, we define the distance between two edges $e,e'$ in $G$ as the minimum distance between two vertices $v,v'$ such that $v,v'$ are incident to $e,e'$, respectively. Thus two incident edges has distance $0$. Given an integer $r\geq 1$, let $ (F, r)${\it-matching} be a collection of copies of $F$ in $G$ such that the distance between any pair of copies of $F$ is at least $r$. Thus $F$-matchings are $(F, 1)$-matchings and induced $F$-matchings are $(F,2)$-matchings. Note that the empty set is also an $(F,r)$-matching for any $F$. Let $S_r(F,G)$ be a set of $(F, r)$-matchings and $$s_r(F, G) = |S_r(F, G)|.$$ To shorten the notations, we will refer to a $(K_2, r)$-matching as a \emph{$r$-matching} and a $(K_2,2)$-matching as \emph{an induced matching}. \\

We are interested in determining the maximum and minimum number of $r$-matchings in a tree of fixed order. The following observation shows that finding the minimum is easy. Hence, we focus on estimating the maximum number of $r$-matchings in a tree of fixed order. Let $\mathbf{T}_{n}$ denote a family of all trees with $n-1$ edges and $n$ vertices, and $P_n$ denote an $n$-vertex path.
\begin{obs}\label{observation}
For $T \in \mathbf{T}_{n}$,  we have $\displaystyle \min_{T\in \mathbf{T}_{n}}s_r(K_2,T) = n$, and the equality only holds for trees $T$ with diameter at most $r+1$.
\end{obs}
\begin{proof}
Since the empty set and sets of one edge are always $r$-matchings,  we have $\displaystyle \min_{T\in \mathbf{T}_{n}}s_r(K_2,T) = n$. 
It is obvious that a tree $T$ has exactly $n$ $r$-matchings if the diameter of $T$ is at most $r+1$. On the other hand, a tree $T$ satisfying the equality must have no two edges with distance at least $r$. Thus the tree $T$ with the equality must have diameter at most $r+1$.\end{proof}

In Section \ref{path}, we approximate the number of $r$-matchings in  $P_n$. In Section \ref{induced}, we prove that the path $P_n$ contains the most number of induced matchings among all trees with $n$ vertices. In Section \ref{big r}, we prove that 
$$  (e^{-\frac{6}{r^2}}\alpha_r)^{n-1} \leq \max_{T \in \mathbf{T}_{n}} s_r(K_2,T) \leq (e^{\frac{1}{r}}\alpha_r)^{n-1} $$ for large $r$ with
$\alpha_r = (s+1)^{\frac{1}{r/2+s}} = e^{\frac{1}{s+1}}$, where $s$ is a real number satisfying $r/2+s = (s+1)\log(s+1)$. We also prove that there are $n$-vertex trees having more $r$-matchings than $P_n$ when $r\in \mathbb{N} \backslash \{1,2,3,4,5,7,9\}$.

\section{ The number of $r$-matchings in $P_{n}$. } \label{path}
In this section, we discuss the asymptotic behavior of $s_r(K_2, P_n)$ as $n$ $\to \infty$. For the sake of formality, we let $s_r(K_2,P_{n})=1$ for $n = 0$.
Take an end edge $e$ and count all $r$-matchings containing $e$ and $r$-matchings not containing $e$. Then we have the following recurrence relations about $s_r (K_2 , P_{n} )$: \\
\begin{equation} 
\label{initial values}
s_r (K_2 , P_{n} ) = n \text{ for } 1 \leq n \leq r,
\end{equation}
\begin{equation} \label{recurrence}
s_r (K_2 , P_{n} )  = s_r (K_2 , P_{n-1} )  + s_r (K_2 , P_{n-r-1} ) \text{ for } n \geq r+1.
\end{equation} \\
\noindent Let $p_r(x) = x^{r+1} - x^{r} - 1 \in \mathbb{C}[x]$ be the characteristic polynomial for the recurrence relation (\ref{recurrence}). Since $p_r(x)$ and its derivative $p'_r(x)$ do not have any common root in $\mathbb{C}$, the polynomial $p_r(x)$ is separable. Let $q_1 , q_2 , ... , q_{r+1}$ be $r+1$ distinct roots of $p_r(x)$ in $\mathbb{C}$, where $|q_1| \geq |q_2| \geq ... \geq |q_{r+1}|$. In this paper, we let log denote the natural logarithm. \\

\begin{thm}\label{lem:lemma1} The polynomial $p_r(x)$ has an unique positive real zero $1 < \beta_r < 2$ which satisfies the following: there exists a constant  $C_r >0 $ such that 
$$C_r = \lim_{n \to \infty}{\frac{s_r(K_2,P_n)}{\beta_r^{n-1}}} = \frac{\beta_r^{2r}}{\beta_r^{r} + (r+1)} = (1+o(1))\frac{r}{(\log{r})^2}.$$\end{thm}

\begin{proof}

Since $p_r(1) < 0$ and $p_r(2) > 0$, there are at least one root of $p_r(x)$ between $1$ and $2$. 
Because of $p'_r(x) = x^{r-1} ((r+1)x - r)>0$, $p_r(x)$ is increasing for $x>1 \geq \frac{r}{r+1}$. It follows that there is an unique positive real root $\beta_r$ of $p_r(x)$ and it lies between $1$ and $2$. 
If $q$ is a non-real root of $P_r(x)$, then $q^{r+1}$ and $q^r$ have the different argument, so  we have $1 =\beta_r^{r+1}-\beta_r^r =|q^{r+1}-q^{r}|> |q^{r+1}|-|q^{r}| =|q|^{r+1}-|q|^r$, which means that  $\beta_{r}> |q|$. Thus $\beta_r$ is the unique root with largest radius since $P_r(x)$ is separable.

First, we prove that  $\beta_r^{r} = (1+o(1))\frac{r}{\log{r}}.$ It is enough to verify the following two inequalities hold for large $r$. 
\begin{equation}\label{eq1} p_r \left((\frac{r}{\log{r}})^{1/r} \right)< 0 
\end{equation}  
\begin{equation}\label{eq2} p_r \left( ((1+ \frac{1}{\sqrt{\log r}})\frac{r}{\log{r}})^{1/r} \right)>0
\end{equation} 

\noindent The inequality (\ref{eq1}) comes from
$$ \log(r)( p_r \left((\frac{r}{\log{r}})^{1/r} \right)+1)= r((\frac{r}{\log{r}})^{1/r}-1) \leq r(e^{\frac{1}{r}(\log{r}-\log{\log{r}})}-1)$$ $$ \leq 
r\left( \frac{\log r - \log \log r}{r} + O(\frac{(\log r)}{r^2}) \right) = \log{r}-\log\log{r}+o(1) < \log{r}.$$

\noindent We also get the inequality (\ref{eq2}) from

$$ \log(r)(p_r \left( ((1+ \frac{1}{\sqrt{\log{r}}})\frac{r}{\log{r}})^{1/r} \right)+1)\geq r(1+\frac{1}{\sqrt{\log{r}}})(( \frac{r}{\log{r}})^{1/r}-1) $$ $$ \geq  r(1+\frac{1}{\sqrt{\log{r}}})(e^{\frac{1}{r}(\log{r}-\log{\log{r}})}-1) \geq 
 r(1+\frac{1}{\sqrt{\log{r}}})( \frac{\log r - \log \log r }{r} + O(\frac{(\log{r})^2}{r^2}))$$ $$ \geq \log{r}+\sqrt{\log{r}}+o(\sqrt{\log{r}}) > \log{r}.$$

\noindent  Now we show that there exists $C_r > 0$ such that $\lim_{n \to \infty}{\frac{s_r(K_2,P_n)}{\beta_r^{n-1}}} = C_r$. 
We use induction on $n$. Because of $s_r(K_2,P_n)>0$, there exists $c>0$ such that $s_r(K_2,P_n)\geq c\beta_r^{n-1}$ for $0 \leq n \leq r$. If this holds for all values less than $n$, then we have $s_r(K_2,P_n) =s_r(K_2,P_{n-1}) + s_r(K_2,P_{n-r-1}) \geq c(\beta_r^{n-2} + \beta_r^{n-r-2}) = c\beta_r^{n-1}$, it also holds for $n$. Thus there exists such $c>0$. Since $p(x)$ is separable, there exist $b_1 \in \mathbb{R}$ and $b_2 , ... , b_{r+1} \in \mathbb{C}$ such that $s_r(K_2,P_n) = b_1 \beta_r^{n-1}+ \sum_{i=2}^{r+1}{b_i {q_i}^{n-1}}$ for $n \geq 0$ and $b_1 > 0$. Then we have $$C_r = \lim_{n \to \infty}{\frac{s_r(K_2,P_n)}{\beta_r^{n-1}}}= b_1 > 0.$$\\

\noindent  Finally we find the value of $C_r$. By considering $r$-matchings in $P_{2n}$ containing $n-r+i$ th edges for $i=0, 1, \cdots,r$ and $r$-matchings not containing any of them, we get the following  
$$s_r(K_2,P_{2n}) = s_r(K_2,P_{n-r})s_r(K_2,P_{n}) + \sum_{i=0}^{r}s_r(K_2,P_{n-2r+i})s_r(K_2,P_{n-i})$$
$$ = (1+o(1))({C_r}^2 \beta_r^{2n-r-2} + (r+1){C_r}^2 \beta_r^{2n - 2r-2}).$$ 

 \noindent Then we have 

\begin{align*}
C_r =\lim_{n \to \infty}\frac{s_r(K_2,P_n)^2}{s_r(K_2,P_{2n-1})} & = \lim_{n \to \infty}\frac{(1+o(1))C_r^2 \beta_r^{2n-2}}{(1+o(1))({C_r}^2 \beta_r^{2n-r-2} + (r+1){C_r}^2 \beta_r^{2n - 2r-2})} \\  & = \frac{\beta_r^{2r}}{\beta_r^{r} + (r+1)}= (1+o(1))\frac{r}{(\log{r})^2}.
\end{align*}\end{proof}

\section{Trees with the maximum number of induced matchings} \label{induced}
In this section, we prove that the path of order $n$ contains the largest number of induced matchings among all trees of order $n$. Before that, we prove the following lemma.

\begin{lem} \label{lem}
Let $T$ be an $n$-vertex tree which is not a path. Take a minimal subtree $T_0$ containing all vertices of degree at least three. Let $v$ be a leaf of $T_0$ and $d(v)=d\geq 3$. Let $v^1,v^2,\cdots,v^{d-1}$ be its neighbors of degree two, each of which belongs to path $P^1, P^2, \cdots, P^{d-1}$ in $T-v$, respectively. Let $|P^i|=p_i$ where $ 1 \leq i \leq d-1$ and $p_1\geq p_2 \geq \cdots \geq p_{d-1}$. Then one of the followings holds. \\

1) $p_1\geq r+1$

2) $p_1\leq r$ and $p_1+p_j >r+1$ for all $2\leq i\neq j\leq d-1$

3) There exists an $n$-vertex tree $T'$ having as many $r$-matchings as $T$ such that the number of leaves in $T'$ is one less than $T$. If $n\geq r+3$ and $T$ has exactly three leaves, then $T'$ has strictly more $r$-matchings than $T$.
\end{lem}
\begin{proof}{ Suppose that  1) and 2) are not true.
Because of the choice of $v$, we know that $d$ is at least $3$ and $p_1+ p_2\leq r+1$. Let $P^1=v^1_1\cdots v^1_{p_1}$ and $P^2=v^2_1\cdots v^2_{p_2}$ and $v^1_1,v^2_1$ be the vertices adjacent to $v$. We construct a new tree $T'$ by replacing $vv^2_1$ with $v^2_1v^1_{p_1}$. 
It is obvious that $T'$ contains exactly one less leaf than $T$.
Let $L= E(P_1)\cup E(P_2) \cup vv^1_1 \cup vv^2_1 \cup v^2_1v^1_{p_1}.$
For any $r$-matching $M$ of $T$, $M \cap L$ has only one element because any two edges in $L$ have distance at most $r$. Hence, for any two edges $e,e'$ in $M$,  the distance between them in $T$ is the same as their distance in $T'$ unless one of them is in $L$. If $e$ is in $L$, the distance between $e$ and $e'$ in $T'$ is at least the distance between them in $T$. Thus, $M$ is still an $r$-matching in $T'$ and $T'$ has as many $r$-matchings as $T$. If $T$ has exactly three leaves, then $T'$ is a path with at least $r+3$ vertices, then we take $\{ v_{p_2-1}^2v_{p_2}^2,e\}$ where an edge $e$ has distance exactly $r+1$ from $v_{p_2-1}^2v_{p_2}^2$. This is an $r$-matching in $T'$, but not an $r$-matching in $T$. Thus $T'$ has strictly more $r$-matchings than $T$. 
}\end{proof}
 
\begin{claim} \label{cl}
For $n \in \mathbb{N}$, we have $s_2(T_2,P_n) \geq 2 s_2(T_2,P_{n-2})$, and the equality holds only for $n=2,3,4$.
\end{claim}
\begin{proof}
It is easy to check that the equality holds for $n=2,3,4$.
For $n\geq 5$, the recurrence relation (\ref{recurrence}) implies that
$s_2(T_2,P_n) = s_2(T_2,P_{n-1})+ s_2(T_2,P_{n-3}) = s_2(T_2,P_{n-2})+ s_2(T_2,P_{n-3})+ s_2(T_2,P_{n-4}) 
>  s_2(T_2,P_{n-2})+ s_2(T_2,P_{n-3})+ s_2(T_2,P_{n-5})
= s_2(T_2,P_{n-2})+s_2(T_2,P_{n-2}) = 2s_2(T_2,P_{n-2})$.\end{proof}

\begin{thm}
For a tree  $T$ of order $n$, we have
$$ s_2(K_2,T) \leq s_2(K_2,P_n)$$
and the equality holds only for $T=P_n$ or $T=K_{1,3}$.
\end{thm}
\begin{proof}

\noindent We use induction on $n$. For the base case, we check all the trees with at most $4$ vertices, and conclude that $P_1, P_2, P_3, P_4$ and $K_{1,3}$ are all possible trees and all of them have exactly one $2$-matching.
 Take all $n$-vertex trees with the most number of $2$-matchings which are not paths. Among those trees, we take a tree $T$ with the smallest number of leaves. We may assume $n\geq 5$. Note that $T$ has at least three leaves since it is not a path.
If $T$ has distance at most three, then the theorem holds by Observation \ref{observation}. So we suppose that $T$ is not a double star. Then take a minimal subtree $T_0$ in $T$ containing all vertices of degree at least three. Let $v$ be a leaf of $T_0$. Then $T-v$ has at least two path components. Let $v^1_{1}v^1_{2}\cdots v^1_{p_1}, \cdots,v^j_{1}v^j_{2}\cdots v^j_{p_1}  $ denote $j$ paths in $T-v$, where  $p_1\geq p_2\geq \cdots \geq p_j$, such that $v^j_{1}$ adjacent to $v$. By Lemma \ref{lem}, one of the following three cases holds.\\

\noindent {\bf {Case 1.} $\mathbf{p_1\geq 3}$.}\\
Then the recurrence relation(\ref{recurrence}) and induction hypothesis imply that
\begin{align*}
s_2(K_2, T) & = s_2(K_2, T-v^1_{p_1}) + s_2(K_2, T-v^1_{p_1-2} -v^1_{p_1-1}-v^1_{p_1}) \\ 
&< s_2(K_2,P_{n-1}) + s_2(K_2,P_{n-r-1})= s_2(K_2, P_{n}). 
\end{align*}
\vspace{0.1cm}

\noindent {\bf {Case 2.} $\mathbf{p_1\leq 2}$ {\bf and} $\mathbf{p_i+p_{i'}\geq 4}$ {\bf for} $\mathbf{1\leq i\neq i'\leq j}$.}\\
Then we have $p_1=\cdots=p_j=2$ and
  $$s_2(K_2, T) = s_2(K_2, T-v^1_2) + 2^{j-1}s_2(K_2, T-v^1_1-v^1_2-\cdots-v^j_1- v^j_{2}-v).$$ 
 We get the last term $2^{j-1} s_2(K_2, T-v^1_1-v^1_2-\cdots-v^j_1-v^j_{2}-v)$ by counting  $2$-matchings  containing $v_1^1v_2^1$ according to its intersection with  $\{v^i_{1}v^i_{2} : i=2,\cdots, j \}$.
By Claim \ref{cl}, we have 
\begin{align*}
s_2(K_2, T) 
&\leq s_2(K_2,P_{n-1}) + 2^{j-1} s_2(K_2, P_{n-2j-1})\\
&<s_2(K_2,P_{n-1}) + s_2(K_2, P_{n-3}) = s_2(K_2,P_n).
\end{align*}
\vspace{0.1cm}

\noindent {\bf {Case 3.}  There exists an $n$-vertex tree $T'$ with as many $2$-matchings as $T$ such that $T'$ have one less leaf than $T$.}\\
 By our choice of $T$, it is only possible when $T'$ is a path and $T$ contains exactly three leaves. In this case,  Lemma \ref{lem} implies that $T'$ has strictly more $2$-matchings than $T$ for $n\geq 5$.  \end{proof}

\section{The number of $r$-matchings in a tree for large $r$} \label{big r}
In this section,  we estimate the maximum number of $r$-matchings in a tree $T$ of fixed order, for large $r$.  Let $s\in \mathbb R$ be such that $r/2+s = (s+1)\log(s+1)$. Note that $(x+1)^{\frac{1}{r/2+x}}$ has its maximum value when $x=s$ and the maximum value is $(1+o(1))\frac{r}{2\log{r}}$. Also note that $r>s$ for all $r\geq 2$. Let $\alpha_r = (s+1)^{\frac{1}{r/2+s}} = e^{\frac{1}{s+1}}$.\\

\noindent Before proving Theorem \ref{main}, we do useful calculations in the following Claims.

\begin{claim} \label{computation 1}
$(e^{\frac{1}{r}}\alpha_r)^{r/2+s/2+1} \geq (e^{\frac{1}{r}}\alpha_r)^{r/2+s/2}+1 $
\end{claim}
\begin{proof}
$(e^{\frac{1}{r}}\alpha_r)^{r/2+s/2+1}-(e^{\frac{1}{r}}\alpha_r)^{r/2+s/2} = (e^{\frac{1}{r}}e^{\frac{1}{s+1}} -1)e^{\frac{r/2+s/2}{r}}(s+1)e^{-\frac{s/2}{s+1}} $ \newline
\indent\indent\indent $\geq (\frac{1}{s+1}+\frac{1}{r})(s+1)e^{\frac{r/2+s/2}{r} - \frac{s/2}{s+1}}
 \geq (1+ \frac{s+1}{r}) e^{\frac{1}{2(s+1)}}e^{\frac{s}{2r}} \geq 1. $
\end{proof}

\begin{claim} \label{computation 2}
For an integer $w$ such that $w=as$ with $0\leq a \leq 1$, we have
$(e^{\frac{1}{r}}\alpha_r)^{r+w+1}\geq (e^{\frac{1}{r}}\alpha_r)^{r+w}+w$.
\end{claim}
\begin{proof}
$(e^{\frac{1}{r}}\alpha_r)^{r+w+1}-(e^{\frac{1}{r}}\alpha_r)^{r+w} = (e^{\frac{1}{r}}(s+1)^{\frac{1}{r/2+s}})^{r+w}(e^{\frac{1}{r}}(s+1)^{\frac{1}{r/2+s}}-1) $ \newline
\indent\indent\indent  $= e^{1+\frac{w}{r}}(s+1)^{2+\frac{w-2s}{(s+1)\log(s+1)}}(e^{\frac{1}{r}}e^{\frac{1}{(s+1)}}-1) \geq e^{1}(s+1)^2 e^{\frac{w-2s}{s+1}}( \frac{1}{s+1}+\frac{1}{r})$ \newline 
\indent\indent\indent $\geq e^{1+\frac{w-2s}{s+1}}(s+1 + \frac{(s+1)^2}{r}) \geq  e^{a-1}(s+1) \geq a(s+1) \geq w.$ \end{proof}

\begin{thm} \label{main}
For an integer $r\geq 2$ and a tree $T$ with $n$ vertices,  $s_r(K_2, T) $ is at most $(s+1)(e^{\frac{1}{r}}\alpha_r)^{n-1}$.
\end{thm}
\begin{proof}{
We use induction on $n$. For $n\leq r+2$, the theorem holds because of the following equation:  $$n \leq (s+1)(1+ \frac{n-1}{r}+ \frac{n-1}{s+1}) \leq (s+1)(e^{\frac{1}{r}} e^{\frac{1}{s+1}})^{n-1} = (s+1)(e^{\frac{1}{r}}\alpha_r)^{n-1}.$$

 \noindent Assume a counterexample exists with minimum $n$. Then we have $n\geq r+1$, and we take all tree $T$ with $n$ vertices with the maximum number of $r$-matchings, and among those we take the one with the minimum number of leaves. Obviously, $T$ has diameter at least $r+2$.

If $T$ is a path, say $v_1v_2\cdots v_p$ with $p\geq r+3$, then we have the following Case 1. Otherwise we take a minimal subtree $T_0$ containing all vertices of degree at least three, and let $u$ be a leaf of $T_0$. Because of the choice of $u$, $T-u$ consists of at least three components and at most one of them is not a path. Let $v_1\ldots v_p$ and $u_1\ldots u_q$ be two paths such that $v_p$, $u_q$ are both adjacent to $u$ and $v_1,u_1$ are both leaves in $T$ and $p\geq q$. By Lemma \ref{lem}, we have either $p\geq r+1$, $p+q>r+1$, or there exists a tree $T'$ with at many $r$-matchings as $T$ with less leaves.
Because of our choice of $T$, it cannot be the third case. Thus it's one of the first two, and in any case, we have $p+q >r+1$. Let $L_1= \{v_1,v_2,\cdots,v_{\min\{p,r+1\}}\}$.
 \\

\noindent {\bf Case 1. } $\mathbf {p \geq r/2 + s/2+1}$. \\

\noindent First, consider all $r$-matchings containing $v_1v_2$. 
Since deleting $v_1v_2$ from them gives $r$-matchings in $T-L_1$, induction hypothesis implies that there are at most $(s+1)(e^{\frac{1}{r}}\alpha_r)^{n-1-\min\{p,r+1\}}\leq (s+1)(e^{\frac{1}{r}}\alpha_r)^{n-r/2-s/2-2}$ $r$-matchings containing $v_1v_2$, since  $\min\{p,r+1\} \geq r/2+s/2+1$ for $r\geq 2$. Again by induction hypothesis, there are $(s+1)(e^{\frac{1}{r}}\alpha_r)^{n-2}$ $r$- matchings not containing $v_1v_2$.
By  Claim \ref{computation 1}, the number of $r$-matchings in $T$ is at most 
$$(s+1)((e^{\frac{1}{r}}\alpha_r)^{n-2}+(e^{\frac{1}{r}}\alpha_r)^{n-r/2-s/2-2}) \leq (s+1)(e^{\frac{1}{r}}\alpha_r)^{n-1}.$$ 

\vspace{0.3cm}

\noindent {\bf Case 2.} $ \mathbf {q\leq p \leq r/2+s/2}$.\\

\noindent Here we have $L_1= \{v_1,v_2,\cdots,v_{p}\}$. Let $L_2 = \{ u_q, \cdots, u_{p+q-r}\}$. Consider a forest $F=T-L_1-L_2$.  Then $F$ contains two components, one of which is a component containing $u_0$, say $T_0$, and the other one $T_1$ does not contain $u_0$.
Then $T_0$ has at most $p+q-r-2 \leq s-1$ edges and $T_1$ has $m-p-q \leq m-r-2$ edges. Also,
the number of all $r$-matchings in $T$ without $v_1v_2$ is  at most $(s+1)(e^{\frac{1}{r}}\alpha_r)^{n-2}$. If we take an $r$-matching with $v_1v_2$ in $T$, and delete $v_1v_2$, then it is also a  $r$-matching in $F$.
Thus, the number of all $r$-matchings in $T$ with $v_1v_2$ is at most the number of $r$-matchings in $F$. Let $p+q= r+w$ with $w = as \leq s$, where $0\leq a \leq 1$. 
Since the number of $r$-matchings in $F$ is at most $(p+q-r)(s+1)(e^{\frac{1}{r}}\alpha_r)^{n-p-q-2}$, by Claim \ref{computation 2} the number of $r$-matchings in $T$ is at most $$(s+1)((e^{\frac{1}{r}}\alpha_r)^{n-2} + w(e^{\frac{1}{r}}\alpha_r)^{n-r-w-2})\leq (s+1)(e^{\frac{1}{r}}\alpha_r)^{n-1} .$$ }\end{proof}

\noindent In the following Theorem \ref{const},
 we prove that Theorem \ref{main} is pretty close to being tight. 

\begin{thm} \label{const}
For large $r$, there is a tree $T$ with $n$ vertices such that the number of $r$-matchings in $T$ is at least $(e^{-\frac{6}{r^2}}\alpha_r)^{n-1}$.
\end{thm}
\begin{proof} Let $n$ be an integer such that $n-1$ is a multiple of $\lceil r/2+s+\frac{1}{2}\rceil$.  We construct a tree $T$ with $\frac{n-1}{\lceil r/2+s+\frac{1}{2}\rceil}$ leaves and $n$ vertices from $K_{1,\frac{n-1}{\lceil r/2+s +\frac{1}{2}\rceil}}$ by subdividing each edge $\lceil r/2+s +\frac{1}{2}\rceil-1$ times.
We consider the forest $F$ obtained from $T$ by deleting all vertices with distance at most $(r-1)/2$ from the vertex of degree $\frac{n-1}{\lceil r/2+s+\frac{1}{2} \rceil}$.

\noindent All $r$-matchings in $F$ are also $r$-matchings in $T$ since  
the distance between $e$ and $e'$ in $T$ is at least $r$ for edges $e,e'$ in two distinct components of $F$. Since the number of $r$-matchings in each path is  $\lceil r/2 + s+\frac{1}{2} \rceil - \lceil \frac{r}{2} \rceil +1 \geq \lceil s+1 \rceil$, the number of $r$-matchings in $F$ is at least 
$$(\lceil s+1\rceil )^{\frac{n-1}{\lceil r/2+s +\frac{1}{2}\rceil}} \geq (s+1)^{\frac{n-1}{r/2+s}}(s+1)^{\frac{n-1}{\lceil r/2+s+\frac{1}{2}\rceil}-\frac{n-1}{r/2+s}}$$ $$\geq (\alpha_r )^{n-1}(s+1)^{-\frac{3(n-1)}{2(r/2+s)^2} } \geq  (e^{-\frac{6}{r^2}}\alpha_r)^{n-1}.$$ Therefore, the number of $r$-matchings in $T$ is at least $(e^{-\frac{6}{r^2}}\alpha_r)^{n-1}$. \end{proof}

\noindent The ratio between Theorem \ref{main} and Theorem \ref{const} is at most $e^{\frac{1}{r} +\frac{6}{r^2}} = \alpha_r^{\frac{1+o(1)}{2\log{r}}}$.
Note that the number of $r$-matchings in the  $n$-vertex path is  $$(C_r+o(1))\beta_r^{n-1} = (C_r+o(1)) \alpha_r^{(\frac{1}{2}+o(1))(n-1)}.$$ By Theorem \ref{const}, there are trees having more $r$-matchings than a path when $r$ is large. Let $T_{a,b}$ denote the subdivided star obtained from $K_{1,b}$ by subdividing each edge $a-1$ times. In the following table, we indicate the actual example constructed in the same way as in Theorem \ref{const} with more precise choice of parameters.

\begin{center}
\begin{tabular}{c||c|c|c|c|c}
$r$ & $s$            & $\alpha_r$     & $\beta_r$ &  $a$  & $s_r(K_2,T_{a,(n-1)/a})^{\frac{1}{n-1}}$ \\ \hline
2 & $1.7182\ldots$ & $1.4446\ldots$ & $1.4655\ldots$ & 3 & $3^{1/3}=1.4422\ldots$ \\
3 & $2.1809\ldots$ & $1.3693\ldots$ & $1.3802\ldots$ & 5 & $4^{1/5}=1.3195\ldots$ \\
4 & $2.5911\ldots$ & $1.3210\ldots$ & $1.3247\ldots$ & 5 & $4^{1/5}=1.3195\ldots$  \\
5 & $2.9673\ldots$ & $1.2866\ldots$ & $1.2851\ldots$ & 6 & $4^{1/6} = 1.2599\ldots$  \\
6 & $3.3191\ldots$ & $1.2605\ldots$ & $1.2554\ldots$ & 6 & $4^{1/6} = 1.2599\ldots$   \\
7 & $3.6523\ldots$ & $1.2397\ldots$ & $1.2320\ldots$ & 8 & $5^{1/8} = 1.2228\ldots$ \\
8 & $3.9706\ldots$ & $1.2228\ldots$ & $1.2131\ldots$ & 8 & $5^{1/8} = 1.2228\ldots$ \\
9 & $4.2766\ldots$ & $1.2086\ldots$ & $1.1974\ldots$ & 10 & $6^{1/10} = 1.1962\ldots$ \\
10 & $4.5723\ldots$ & $1.1965\ldots$ & $1.1842\ldots$ & 10 & $6^{1/10} = 1.1962\ldots$\\
11 & $4.8592\ldots$ & $1.1861\ldots$ & $1.1729\ldots$ & 11 & $6^{1/11} = 1.1769\ldots$\\
\end{tabular}
\end{center}

\noindent For $r$ larger than $11$, $\beta_r < s_r(K_2,T_{a,n/a})^{\frac{1}{n-1}}$ holds since $\beta_r$ decrease faster than $s_r(K_2,T_{a,n/a})^{\frac{1}{n-1}}$. Hence, we can conclude that $$\beta_r < s_r(K_2,T_{a,n/a})^{\frac{1}{n-1}}$$  for all $r$ except $r\in \{1,2,3,4,5,7,9\}$. Thus, there are $n$-vertex trees with more $r$-matchings than $P_n$ for $r\notin\{1,2,3,4,5,7,9\}$. For $r=1,2$, such tree does not exist, and 
we do not know for $r \in \{3,4,5,7,9\}$. We leave this as an open question.

\begin{problem}\label{prob}
For $r\in \{3,4,5,7,9\}$, does the following hold?
$$\max_{T_n\in \mathbf{T}_n} s_2(K_2,T_n) = s_2(K_2,P_n)$$
\end{problem}

\end{document}